\documentclass[12pt]{amsart}
\usepackage{amscd,amsmath,amsthm,amssymb,graphics}
\usepackage{pstcol,pst-plot,pst-3d}
\usepackage{lmodern,pst-node}
\usepackage[dvips]{graphicx}
\usepackage{multicol}
\usepackage{epic,eepic}
\usepackage{amsfonts,amssymb,amscd,amsmath,enumitem,verbatim}
\usepackage{tikz}
\psset{unit=0.7cm,linewidth=0.8pt,arrowsize=2.5pt 4}

\newpsstyle{fatline}{linewidth=1.5pt}
\newpsstyle{fyp}{fillstyle=solid,fillcolor=verylight}
\definecolor{verylight}{gray}{0.97}
\definecolor{light}{gray}{0.9}
\definecolor{medium}{gray}{0.85}
\definecolor{dark}{gray}{0.6}

\usepackage{lineno}

\def\NZQ{\Bbb}               
\def\NN{{\NZQ N}}

\def\KK{{\NZQ K}}
\def\frk{\frak}               

\def\pp{{\frk p}}

\def\mm{{\frk m}}

\def\Phi{{\frk n}}
\def\Phi{{\frk N}}
\def\MR{{\mathcal R}}

\def\MF{{\mathcal F}}

\def\MC{{\mathcal C}}

\def\ab{{\bold a}}

\def\xb{{\bold x}}

\def\opn#1#2{\def#1{\operatorname{#2}}} 
\opn\chara{char} \opn\length{\ell} \opn\pd{pd} \opn\rk{rk}
\opn\projdim{proj\,dim} \opn\injdim{inj\,dim} \opn\rank{rank}
\opn\depth{depth} \opn\grade{grade} \opn\height{height}
\opn\embdim{emb\,dim} \opn\codim{codim}

\opn\Tr{Tr} \opn\bigrank{big\,rank}
\opn\superheight{superheight}\opn\lcm{lcm}
\opn\trdeg{tr\,deg}
\opn\reg{reg} \opn\lreg{lreg} \opn\ini{in} \opn\lpd{lpd}
\opn\size{size}\opn\bigsize{bigsize}
\opn\cosize{cosize}\opn\bigcosize{bigcosize}
\opn\sdepth{sdepth}\opn\sreg{sreg}
\opn\link{link}\opn\fdepth{fdepth}
\opn\div{div} \opn\Div{Div} \opn\cl{cl} \opn\Cl{Cl}
\let\epsilon\varepsilon
\let\phi=\varphi
\let\kappa=\varkappa
\opn\Spec{Spec} \opn\Supp{Supp} \opn\supp{supp} \opn\Sing{Sing}
\opn\Ass{Ass} \opn\Min{Min}\opn\Mon{Mon} \opn\dstab{dstab} \opn\astab{astab}
\opn\Syz{Syz}
\opn\Ann{Ann} \opn\Rad{Rad} \opn\Soc{Soc}
\opn\Im{Im} \opn\Ker{Ker} \opn\Coker{Coker} \opn\Am{Am}
\opn\Hom{Hom} \opn\Tor{Tor} \opn\Ext{Ext} \opn\End{End}
\opn\Aut{Aut} \opn\id{id}

\opn\nat{nat}
\opn\pff{pf}
\opn\Pf{Pf} \opn\GL{GL} \opn\SL{SL} \opn\mod{mod} \opn\ord{ord}
\opn\Gin{Gin} \opn\Hilb{Hilb}\opn\sort{sort}
\opn\initial{init}
\opn\ende{end}
\opn\height{height}
\opn\type{type}
\opn\set{set}
\opn\aff{aff} \opn\con{conv} \opn\relint{relint} \opn\st{st}
\opn\lk{lk} \opn\cn{cn} \opn\core{core} \opn\vol{vol}
\opn\link{link} \opn\star{star}\opn\lex{lex}
\opn\gr{gr}

\def\pot#1#2{#1[\kern-0.28ex[#2]\kern-0.28ex]}

\opn\dirlim{\underrightarrow{\lim}}
\opn\inivlim{\underleftarrow{\lim}}
\def\Implies{\ifmmode\Longrightarrow \else
        \unskip${}\Longrightarrow{}$\ignorespaces\fi}
\def\implies{\ifmmode\Rightarrow \else
        \unskip${}\Rightarrow{}$\ignorespaces\fi}
\def\iff{\ifmmode\Longleftrightarrow \else
        \unskip${}\Longleftrightarrow{}$\ignorespaces\fi}

\let\:=\colon
 \theoremstyle{plain}
\newtheorem{Theorem}{Theorem}[section]
 \newtheorem{Lemma}[Theorem]{Lemma}
 \newtheorem{Corollary}[Theorem]{Corollary}
 \newtheorem{Proposition}[Theorem]{Proposition}

 \theoremstyle{definition}
 \newtheorem{Definition}[Theorem]{Definition}
 \newtheorem{Remark}[Theorem]{Remark}
 
 \newtheorem{Example}[Theorem]{Example}

\let\epsilon\varepsilon
\let\kappa=\varkappa
\opn\dis{dis}
\def\pnt{{\raise0.5mm\hbox{\large\bf.}}}

\opn\Lex{Lex}

\begin{document}
\title{Almost normally torsionfree ideals}
\author {Claudia Andrei-Ciobanu}

\address{Claudia Andrei-Ciobanu, Faculty of Mathematics and Computer Science, University of Bucharest, Str. Academiei 14,
 010014 Bucharest, Romania} \email{claudiaandrei1992@gmail.com}

\thanks{The author would like to thank Professor Viviana Ene for valuable suggestions and comments during the preparation of this paper.}

\begin{abstract}
We describe all connected graphs whose edge ideals are almost normally torsionfree. We also prove that the facet ideal of a special odd cycle is almost normally torsionfree. Finally, we determine the $t$-spread principal Borel ideals generated in degree $3$ which are almost normally torsionfree.
\end{abstract}


\maketitle
\section*{Introduction}

Combinatorial algebra is closely related to monomial ideals. A special class of monomial ideals are squarefree monomial ideals. They can be associated with simplicial complexes or graphs. In this paper, we study squarefree monomial ideals which are very near to normally torsionfree ideals.

Let $\KK$ be a field and $S=\KK[x_1,\ldots, x_n]$ be the polynomial ring in $n$ variables over $\KK$. If $I$ is a monomial ideal in $S$, then it is known that the associated prime ideals of $I$ are generated by variables; see \cite[Section~1.3]{HHBook}.
In \cite[Lemma~2.3]{HRV} and \cite[Lemma~2.11]{FHV}, it was given a method for checking whether a prime ideal $P$ is an associated prime ideal of $I$. More precisely, $P\in \Ass(I)$ if and only if $\depth(S_P/I_P)=0$, where $S_P$ is the polynomial ring $\KK[\{x_i: x_i\in P\}]$ and $I_P$ is the localization of $I$ with respect to $P$.

An important result of Brodmann from \cite{BR} shows in particular, that the set of associated prime ideals of a monomial ideal $I\subset S$ stabilizes which means that there exists an integer $k_0$ such that $\Ass(I^k)=\Ass(I^{k_0})$ for all $k\geq k_0$. A monomial ideal $I$ of the polynomial ring $S$ satisfies the persistence property if \[\Ass(I)\subset\Ass(I^2)\subset\ldots\subset\Ass(I^k)\subset\Ass(I^{k+1})\subset\ldots.\] Classes of monomial ideals which have the persistence property are studied in \cite{AEL}, \cite{BMV}, \cite{FHV} and \cite{HRV}. For example, in \cite[Corollary 2.6]{AEL}, it was shown that a $t$-spread principal Borel ideal has the persistence property, while in \cite[Lemma 2.12]{BMV}, it was proved that the edge ideal of a simple graph has the persistence property.

In \cite[Section 1.4]{HHBook}, normally torsionfree ideals are studied. A monomial ideal in $S$ is called normally torsionfree if $\Ass(I^k)=\Min(I)$ for all $k\geq 1$. There it was shown that the normally torsionfree squarefree monomial ideals have the property that their symbolic powers coincide with the ordinary powers. Moreover, in \cite[Theorem~10.3.16]{HHBook}, it was given a combinatorial characterization of normally torsionfree facet ideals of simplicial complexes.

In this paper, we generalize normally torsionfree ideals and we introduce almost normally torsionfree ideals. A monomial ideal $I\subset S$ is called almost normally torsionfree if there exist a positive integer $k$ and a monomial prime ideal $P$ such that $\Ass(I^m)=\Min(I)$ for $m\leq k$ and $\Ass(I^m)\subseteq\Min(I)\cup\{P\}$ for $m\geq k+1.$ This definition is slightly more general than that considered in \cite{HLR}, where the monomial prime ideal $P$ is only given by the graded maximal ideal $\mm=(x_1,\ldots, x_n)\subset S$. Notice that any normally torsionfree ideal is almost normally torsionfree.

In the first section, we refer to edge ideals. According to \cite[Theorem~5.9]{SVV}, it is known that the edge ideal of a simple graph $G$ is normally torsionfree if an only if $G$ is a bipartite graph. The aim of the main result of this section is to classify all connected graphs whose edge ideals are almost normally torsionfree; see Theorem~\ref{graph}.

The second section is dedicated to facet ideals of special odd cycles. We have seen before that \cite[Theorem 10.3.16]{HHBook} gives a large class of normally torsionfree facet ideals by using special odd cycles. In Theorem~\ref{specialcycle}, we prove that the facet ideal of a special odd cycle is almost normally torsionfree.

In the last section, we study $t$-spread principal Borel ideals generated in degree $3$. They have been recently introduced in \cite{EHQ}. The main result of this section is Theorem~\ref{tspread}, where we characterize the $t$-spread principal Borel ideals generated in degree $3$ which are almost normally torsionfree. Moreover, we prove that in this case, the properties normally torsionfree and almost normally torsionfree coincide. We notice that for degree $2$, the behaviour of the set of associated prime ideals of the powers of a $t$-spread principal Borel ideal was given in \cite[Theorem 1.1]{LR}. We recover here this result by a different proof; see Proposition~\ref{deg2}.

\section{Almost normally torsionfree edge ideals}

In this section, we present the behaviour of the set of associated prime ideals of the powers of an edge ideal.

Let $G$ be a finite simple graph on the vertex set $V(G)=[n]$. Recall that a \textit{walk of length $r$} in $G$ is a sequence of edges \[\gamma=(\{i_0,i_1\},\{i_1, i_2\},\ldots, \{i_{r-1}, i_r\})\] with each $i_k\in [n]$ for $k\in\{0,1,\ldots, r\}$. An \textit{even walk} (respectively \textit{odd walk}) is a walk of even (respectively odd) length. A \textit{closed walk} is a walk with $i_0=i_r$ and a \textit{cycle} is a closed walk of the form \[C=(\{i_0,i_1\},\{i_1, i_2\},\ldots, \{i_{r-1}, i_0\})\] with $i_j\neq i_k$ for all $0\leq j<k\leq r-1$.

The graph $G$ is a \emph{bipartite graph} if it does not contain any odd cycle.

Let $S=\KK[x_1,\ldots,x_n]$ be the polynomial ring over the field $K$. The \textit{edge ideal} of $G$ is the monomial ideal
\[I(G)=(x_ix_j: \{i,j\}\text{ is an edge in }G)\subset S.\]

By \cite[Lemma~9.1.4]{HHBook}, every minimal prime ideal of $I(G)$ is given by a minimal vertex cover of $G$. A vertex cover of $G$ is a subset $A\subset[n]$ such that $\{i,j\}\cap A\neq \emptyset$ for all edges $\{i,j\}$ of $G$. A vertex cover $A$ is called minimal if no proper subset of $A$ is a vertex cover of $G$.
In \cite[Theorem~3.3]{CMS}, it was given a method to construct associated prime ideals of the powers of $I(G)$ by considering the odd cycles of $G$.

Let $I\subset S$ be a monomial ideal and $P\subset S$ be a monomial prime ideal containing $I$. Then $P$ is a \textit{monomial persistent prime ideal} of $I$, if whenever $P\in \Ass(I^k)$ for some $k$, then $P\in\Ass(I^{k+1})$. The set of monomial persistent prime ideals is denoted by $\Ass^{\infty}(I)$.

According to \cite[Lemma~2.12]{BMV}, $I(G)$ has the \textit{persistence property} which means that the sets of associated primes of powers of $I(G)$ form an ascending chain. In other words, all monomial prime ideals $P\in\bigcup_{k>0}\Ass(I(G)^k)$ are monomial persistent prime ideals.

In \cite[Theorem~5.9]{SVV}, it was shown that $I(G)$ is normally torsionfree if and only if $G$ is a bipartite graph. Recall that a monomial ideal $I\subset S$ is \textit{normally torsionfree} if $\Min(S/I)=\Ass(S/I^k)$ for all $k\geq 1$.

\begin{Definition}
A monomial ideal $I\subset S$ is called \emph{almost normally torsionfree} if there exist a positive integer $k$ and a monomial prime ideal $P$ such that
$\Ass(I^m)=\Min(I)$ for $m\leq k$ and $\Ass(I^m)\subseteq\Min(I)\cup\{P\}$ for $m\geq k+1.$
\end{Definition}

In particular, any normally torsionfree ideal is almost normally torsionfree.

In this section we study edge ideals which are almost normally torsionfree. The following lemma will allow us to restrict our attention to connected graphs. It generalizes \cite[Lemma 3.4]{HM}. Notice that its proof is the same with that given in \cite{HM}, but for the convenience of the reader, we will include it here.

Let $I$ be a monomial ideal of $S$. From now on, $G(I)$ denotes the minimal set of monomial generators of $I$.

\begin{Lemma}\label{split}
Let $I_1\subset S_1=\KK[x_1,\ldots,x_n], I_2\subset S_2=\KK[y_1,\ldots,y_m]$ be two monomial ideals in disjoint sets of variables.
Let \[I=I_1S+I_2S\subset S=\KK[x_1,\ldots,x_n,y_1,\ldots,y_m].\] Then $P\in \Ass_S(I^k)$ if and only if $P=P_1S+P_2S$ where
$P_1\in \Ass_{S_1}(I_1^{k_1})$ and $P_2\in \Ass_{S_2}(I_2^{k_2})$ for some positive integers $k_1,k_2$ with $k_1+k_2=k+1.$
\end{Lemma}
\begin{proof}
  Let $P\in \Ass_S(I^k)$. Then $P=P_1S+P_2S$, where $P_1=P\cap S_1$ and $P_2=P\cap S_2$. Moreover, $P=I^k:u$ for some monomial $u\in I^{k-1}\setminus I^k$. We consider $u=u_1u_2$, where $u_1$ and $u_2$ are monomials in $S_1$, respectively $S_2$. It implies that $u_1\in I_1^{l_1}\setminus I_1^{l_1+1}$ and $u_2\in I_2^{l_2}\setminus I_2^{l_2+1}$ for some $l_1, l_2\in\{0,1,\ldots, k-1\}$ and $l_1+l_2=k-1$. Now we prove that \[P_1=I_1^{l_1+1}:u_1\text{ and }P_2=I_2^{l_2+1}:u_2.\] If $v_1\in G(P_1)$ and $v_2\in G(P_2)$, then $v_1, v_2\in G(P)$.  Thus, $v_1u, v_2u\in I^k$. Since $v_1u=(v_1u_1)u_2$ and $v_2u=u_1(v_2u_2)$, we obtain $v_1u_1\in I^{k-l_2}=I^{l_1+1}$ and $v_2u_2\in I^{k-l_1}=I^{l_2+1}$. In other words, $v_1\in I_1^{l_1+1}:u_1$ and $v_2\in I_2^{l_2+1}:u_2$. Now let $w_1\in I_1^{l_1+1}:u_1$ and $w_2\in I_2^{l_2+1}:u_2$ and we may suppose that $w_1\in S_1$ and $w_2\in S_2$. Then $w_1u=w_1u_1u_2\in I^{l_1+1+l_2}=I^k$ and $w_2u=u_1w_2u_2\in I^{l_1+l_2+1}=I^k$. Therefore, $w_1\in P\cap S_1=P_1$ and $w_2\in P\cap S_2=P_2$.

  We conclude that $P_1=I_1^{k_1}:u_1$ and $P_2=I_2^{k_2}:u_2$, where $k_1=l_1+1$ and $k_2=l_2+1$. Notice that $k_1+k_2=l_1+1+l_2+1=k+1$.

  Conversely, let $P_1\in \Ass_{S_1}(I_1^{k_1})$ and $P_2\in \Ass_{S_2}(I_2^{k_2})$ for some positive integers $k_1,k_2$ with $k_1+k_2=k+1$. We consider $P_1=I^{k_1}:u_1$ for some monomials $u_1\in I_1^{k_1-1}\setminus I^{k_1}$ and $P_2=I^{k_2}:u_2$ for some monomials $u_2\in I_2^{k_2-1}\setminus I^{k_2}$. Notice that $u_1u_2\in I^{k_1-1+k_2-1}=I^{k-1}$ and $u_1u_2\notin I^k$ because $u_1\in S_1$ and $u_2\in S_2$. In what follows, we show that \[P=I^k:u_1u_2.\]

  If $u\in G(P)$, then $u\in G(P_1)$ or $u\in G(P_2)$. In other words, $uu_1u_2\in I_1^{k_1}I_2^{k_2-1}\subset I^k$ or $uu_1u_2\in I_1^{k_1-1}I_2^{k_2}\subset I^k$. Hence, we have $u\in I^k:u_1u_2$ and $P\subseteq I^k:u_1u_2$.

  If $v\in I^k:u_1u_2$, then we consider $v=v_1v_2$ with $v_1\in S_1$ and $v_2\in S_2$. In the case that $v_1u_1\in I_1^{k_1-1}\setminus I_1^{k_1}$ and $v_2u_2\in I_2^{k_2-1}\setminus I_2^{k_2}$, we obtain $vu_1u_2=(v_1u_1)(v_2u_2)\notin I^k$, which is false. So, $v_1u_1\in I_1^{k_1}$ or $v_2u_2\in I_2^{k_2}$ which implies that $v_1\in P_1$ or $v_2\in P_2$. Therefore, we obtain $v\in P$ and $I^k:u_1u_2\subseteq P$.
\end{proof}

\medskip

Let $G$ be a  graph on the vertex set $[n]$. For a subset $A\subset [n]$, one denotes $N(A)$ the set of all the vertices of $G$ which are
adjacent to some vertices in $A.$ In other words, $N(A)$ is the  set of all the neighbors of the  vertices in $A.$ In addition, if $A\subset [n]$, $G\setminus A$ denotes the induced subgraph of $G$ on the vertex set $[n]\setminus A.$ If $A=\{j\},$ we write $G\setminus j$ instead of $G\setminus \{j\}.$

\begin{Theorem}\label{graph}
Let $G$ be a connected graph on the vertex set $[n]$ and $I=I(G)$ the edge ideal of $G$. Suppose that $G$ contains an odd cycle and
let \[k=\min\{j: G \text{ has an odd cycle } C_{2j+1}\}.\]
Then $\Ass^{\infty}(I)=\Min(I)\cup\{\mm\}$ if and only if, for every odd cycle $C$ of $G,$ we have $V(C)\cup N(C)=[n].$
Moreover, if the above condition holds, then $\Ass(I^m)=\Min(I)$ if $m\leq k$ and $\Ass(I^m)=\Min(I)\cup \{\mm\}$ if $m\geq k+1.$
\end{Theorem}

\begin{proof}
Let $\Ass^{\infty}(I)=\Min(I)\cup\{\mm\}$ and assume that there exists an odd cycle $C=C_{2j+1}$ for some $j\geq 1$
such that $V(C)\cup N(C)\subsetneq [n].$ By \cite[Theorem 3.3]{CMS}, it follows that \[P=(x_i:i\in V(C)\cup N(C)\cup W)\] is an associated
prime ideal of $I^m$ for $m\geq j+1$, where $W$ is any minimal subset of $[n]$ for which $V(C)\cup N(C)\cup W$ is a vertex cover of $G.$
If $V(C)\cup N(C)$ is a vertex cover of $G,$ then $P=(x_i:i\in V(C)\cup N(C))$ is an associated prime of $I^m$ for $m\geq j+1$ with
$P\not\in \Min(I)$ and $P\neq \mm.$ This is a contradiction to our hypothesis. If $V(C)\cup N(C)$ is not a vertex cover of $G,$ then
there exists at least two minimal subsets $W_1,W_2$ of $[n]$ for which $V(C)\cup N(C)\cup W_1$ and $V(C)\cup N(C)\cup W_2$ are vertex covers of $G.$ Therefore, we get at least two minimal prime ideals \[P_1=(x_i:i\in V(C)\cup N(C)\cup W_1)\text{ and }
P_2=(x_i:i\in V(C)\cup N(C)\cup W_2)\] which are not minimal primes of $I$ and are associated with $I^m$ for $m\geq j+1,$ a contradiction.

Let $V(C)\cup N(C)=[n]$ for every odd cycle of $G.$ In particular, this equality holds for the cycle $C=C_{2k+1}$ with the smallest number of vertices.  First we show that $\mm\in \Ass(I^m)$ for $m\geq k+1.$ Since, by \cite[Theorem 2.15]{BMV}, $I=I(G)$ has the persistence property, it is enough to show that $\mm\in \Ass(I^{k+1}).$ Let $u=\prod_{i\in V(C)}x_i$ be the product of all the variables of cycle
$C.$ Clearly, $u\in I^k$ and since $\deg(u)=2k+1,$ we have $u\notin I^{k+1}$ which implies that $I^{k+1}:u\subseteq \mm.$ It is also easily seen that
 for all  $j\in V(C)\cup N(C), x_j u\in I^{k+1} $. Therefore, we have $\mm\subseteq I^{k+1}:u $ since $V(C)\cup N(C)=[n].$

Further on, we show that $\mm\not\in \Ass(I^m)$ for $m\leq k.$ We proceed by induction on the number of vertices of $G$. If
$G=C,$ then the claim follows by \cite[Lemma 3.1]{CMS}. Now let $V(C)\subsetneq [n]$ and assume that there exists $m\leq k$ such that
$\mm\in \Ass(I^m).$ By the persistence property of edge ideals, we may assume that $\mm\in \Ass(I^k).$ Thus, there exists a monomial
$v\not\in I^k$ such that $I^k:v=\mm. $ There exists some vertex $j\in V(C)$ such that $x_j$ does not divide $v,$ since, otherwise,
$v\in I^k.$ Let $J$ be the edge ideal of $H=G\setminus j$ and $Q=(x_i: i\neq j).$  Then, by \cite[Lemma 2.6]{CMS}, it follows that $Q\in \Ass(J^m).$
Note that the graph $H$ is either bipartite or it has an odd cycle of length greater than $2k+1.$ If $H$ is bipartite, then $Q\in \Min(J)$.  If $H$ is not bipartite,  the inductive  hypothesis implies  that $Q\in \Min(J).$ But this is not possible since the set
$[n]\setminus \{j\}$ is not a minimal vertex cover of $H.$

What is left is to show that, in our hypothesis, if $P\in \Ass(I^m) $ for some integer $m\geq 1$ and $P\neq \mm,$ then $P\in \Min(I).$
Let $P\in \Ass(I^m) $ for some integer $m\geq 1$ with $P\neq \mm.$ We proceed again by induction on the number of vertices of $G.$ The claim is obvious if $G=C$ by \cite[Lemma 3.1]{CMS}. Let $V(C)\subsetneq [n]$. Then there exists a vertex $j\in V(G)$ such that
$x_j\not\in P.$ We localize at the prime ideal $Q=(x_i:i\neq j).$ Then $I_Q=(I_1,I_2)$ where $I_1$ is generated by all the variables
$x_i$ with $i\in N(j)$ and $I_2=I(G\setminus N[j])$. Here $N[j]$ denotes the set $N(j)\cup\{j\}.$ By \cite[Corollary 2.3]{CMS}, it
follows that $P_Q=(I_1,P_2)$ where $P_2\in \Ass(I_2^{m^\prime})$ for some $m^\prime \leq m.$ If $G\setminus N[j]$ is bipartite, then $P_2$
gives a minimal vertex cover of $G\setminus N[j]$ and, thus, $P$ gives a minimal vertex cover of $I,$ equivalently, $P\in \Min(I).$
Let $G\setminus N[j]$ be not bipartite. If it is connected, by induction it follows that $P_2$ is a minimal prime ideal of $I_2$ and, consequently, $P\in \Min(I).$ If $G\setminus N[j]$ is not connected, say it has the connected components $G_1,\ldots,G_c$ for some
$c\geq 2,$ then $P_2$ is a sum of minimal prime ideals of $I(G_1),\ldots, I(G_c),$ and again, it follows that $P\in \Min(I).$
\end{proof}

\begin{Remark}
Let $G$ be a graph on the vertex set $[n]$ and $I=I(G)\subset S$ be its edge ideal. Condition $V(C)\cup N(C)=[n]$, for every odd cycle $C$ of $G$, in the hypothesis of previous theorem needs to be checked for every odd cycle $C$ of $G$. Indeed, if $G$ is the graph of Figure~\ref{fig 0}, then \[|\Ass(I(G)^2)|=13>8=|\Min(I(G))|.\]
\begin{figure}
\centering
  \begin{tikzpicture}[domain=0:9]
        \draw[black] (0,0) node[below] {$4$} -- (4,0) node[below] {$3$} -- (4,4) node[above] {$2$} -- (0,4) node[above] {$1$} -- (0,0);
        \draw[black] (0,0) -- (1,1) node[left] {$7$} -- (3,1) node[right] {$6$} -- (4,0);
        \draw[black] (0,4) -- (2,3) node[above] {$5$} -- (4,4);
        \draw[black] (1,1) -- (2,3) -- (3,1) -- (2,2) node[below] {$8$} -- (1,1);
        \draw[black] (2,2) -- (2,3);
    \end{tikzpicture}
\caption{$G$}\label{fig 0}
\end{figure}

Notice that $V(C_1)\cup N(C_1)=[8]$ for $C_1=(\{5,6\},\{6,7\},\{7,5\})$, while for $C_2=(\{6,7\},\{7,8\},\{8,6\})$, we have $V(C_2)\cup N(C_2)=\{3,4,5,6,7,8\}\neq [8]$.
\end{Remark}

Let $G$ be a graph on the vertex set $[n]$ and $I=I(G)\subset S$ its edge ideal. The symbolic Rees algebra of $I$ is
\[\MR_s(I)=\oplus_{k\geq 0}I^{(k)}t^k,\] where $I^{(k)}$ is the $k$-th symbolic power of $I.$ This algebra is actually the vertex cover algebra of $G$ which is generated as an $S$-algebra by all the monomials $\xb^\ab t^b$ where $\ab$ is an indecomposable vertex cover
of $G$ of order $b$. Recall from \cite{HHT} that a vector $\ab=(a_1,\ldots,a_n)\in \NN^n$ is a vertex cover of $G$ of order $b$ if $a_i+a_j\geq b$ for all $i,j$ with $\{i,j\}\in E(G).$ A vertex cover $\ab$ of order $b$ is decomposable if there exists a vertex cover $\ab_1$ of order $b_1$ and a vertex cover $\ab_2$ of order $b_2$ such that $\ab=\ab_1+\ab_2$ and $b=b_1+b_2$. The vector $\ab$ is called indecomposable if $\ab$ is not decomposable.

In \cite[Proposition 5.3]{HHT}, it was shown that the graded $S$-algebra $R_s(I)$ is generated by  the monomial $x_1x_2\cdots x_nt^2$ and the monomials $\xb^\ab t$ such that $\ab$ is  a minimal vertex cover of $G$ if and only if for every odd cycle $C$ of $G$ and for every vertex $i\in [n],$ there exists a vertex $j\in C$ such that $\{i,j\}\in E(G).$ This latter condition characterizes the almost normally torsionfree edge ideals. Therefore, we can derive the characterization of almost normally torsionfree edge ideals in terms of their associated symbolic Rees algebras.

\begin{Corollary}\label{Reesgraphs}
Let $G$ be a graph on the vertex set $[n].$ Then the following statements are equivalent:
\begin{itemize}
	\item [(i)]  The edge ideal $I(G)$ is almost normally torsionfree;
	\item [(ii)] The symbolic Rees algebra $\MR_s(I(G))$  is generated as a graded $S$-algebra by the monomial $x_1x_2\cdots x_nt^2$ and the monomials $\xb^\ab t$ such that $\ab$ is  a minimal vertex cover of $G.$
\end{itemize}
\end{Corollary}

\section{Special odd cycles and the associated primes of their powers}

In this section, we study the associated prime ideals of the powers of a facet ideal of a special odd cycle.

Let $S=\KK[x_1,\ldots, x_n]$ be the polynomial ring in $n$ variables over a field $\KK$ and $\Delta$ be a simplicial complex on the vertex set $[n]$ with the facet set $\MF(\Delta).$ For every subset $F\subset [n]$, we set $\xb_F=\prod_{j\in F}x_j$.  Let $I(\Delta)=(\xb_F: F\in \MF(\Delta))\subset S$ be the facet ideal of $\Delta.$

A \textit{cycle of length} $r\geq 2$ in $\Delta$
is an alternate sequence of distinct vertices and facets
\[
\MC: v_1, F_1, v_2, F_2, v_3, \ldots v_r, F_r, v_{r+1}=v_1
\] such that $v_i,v_{i+1}\in F_i$ for $1\leq i\leq r.$  A cycle $\MC$ is called \textit{special} if no facet of it contains more than two vertices of the cycle.

By \cite[Theorem 10.3.16]{HHBook}, if $\Delta$ has no special odd cycle, then the facet ideal $I(\Delta)$ is normally torsionfree, that is,
$\Ass(I(\Delta)^m)=\Min(I(\Delta))$ for $m\geq 1.$

In what follows, we study $\Ass^{\infty}(I)$ when $I=I(\Delta)$ and  $\Delta$ is a special odd cycle.

\begin{Theorem}\label{specialcycle}
Let $\Delta:  v_1, F_1, v_2, F_2, v_3, \ldots v_{2s+1}, F_{2s+1}, v_{2s+2}=v_1$ be a special odd cycle with $s\geq 1.$ Then
$\Ass(I(\Delta)^m)=\Min(I(\Delta))$ for $m\leq s$ and $\Ass(I(\Delta)^m)=\Min(I(\Delta))\cup\{(x_{v_1},\ldots,x_{v_{2s+1}})\}$ for $m\geq s+1.$ In particular, the facet ideal of a special odd cycle is almost normally torsionfree.
\end{Theorem}

\begin{proof}
Without loss of generality, we may assume that the vertices $v_1,\ldots, v_{2s+1}$ of $\Delta$ are labeled as $1,\ldots,2s+1.$
We set $I=I(\Delta)$ and  $P=(x_1,\ldots,x_{2s+1}).$
We split the proof in three steps.

\textit{Step 1.} We show that $P\in \Ass(I^m)$ for $m\geq s+1.$ First, we show that $P\in \Ass(I^{s+1}).$ Let
\[
u=x_1 \xb_{F_1\setminus\{1,2\}}x_2 \xb_{F_2\setminus\{2,3\}}x_3\cdots x_{2s}\xb_{F_{2s}\setminus\{2s,2s+1\}}x_{2s+1}\xb_{F_{2s+1}\setminus
\{2s+1,1\}}.
\]
We claim that $u\not\in I^{s+1}.$ Indeed, if $u\in I^{s+1},$ then there exists $\xb_{F_{j_1}}, \ldots,\xb_{F_{j_{s+1}}}\in I$
with $1\leq j_1\leq \cdots\leq j_{s+1}\leq 2s+1,$ such that $\prod_{k=1}^{s+1}\xb_{F_{j_k}}$ divides $u.$

If there exists some $1\leq j\leq 2s+1$ such that $\xb_{F_j}^2 \mid u,$ then
\[
\xb_{F_j} \mid x_1 \xb_{F_1\setminus\{1,2\}}x_2\cdots x_{j-1}\xb_{F_{j-1}\setminus \{j-1,j\}} \xb_{F_{j+1}\setminus \{j+1,j+2\}}
x_{j+2}\cdots x_{2s+1} \xb_{F_{2s+1}\setminus\{2s+1,1\}}.\]
Clearly, $x_j$ does not divide $\xb_{F_{j-1}\setminus \{j-1,j\}}$, thus there exists $\ell <j-1$ or $\ell\geq j+1$ such that
$F_\ell \ni j.$ But this is a contradiction, since $\Delta$ is a special cycle, thus no facet may contain more than two vertices of the cycle. Therefore, we may assume that $1\leq j_1< \cdots< j_{s+1}\leq 2s+1,$ and $\prod_{k=1}^{s+1}\xb_{F_{j_k}}$ divides $u.$ If
$j_k\geq j_{k-1}+2$ for all $2\leq k\leq s+1,$ we get $j_1=1,j_2=3,\ldots,j_{s+1}=2s+1$ and $\prod_{k=1}^{s+1}\xb_{F_{2k-1}}\mid u.$ Then it follows that $x_1^2\mid u,$ which is false. Therefore, there must be an index $j$ such that $\xb_{F_j}\xb_{F_{j+1}}\mid u.$ This yields
\[\xb_{F_{j+1}}\mid x_1 \xb_{F_1\setminus\{1,2\}}x_2\cdots x_{j-1}\xb_{F_{j-1}\setminus \{j-1,j\}} \xb_{F_{j+1}\setminus \{j+1,j+2\}}
x_{j+2}\cdots x_{2s+1} \xb_{F_{2s+1}\setminus\{2s+1,1\}}\]
which implies that $x_{j+1}$ divides the right side  monomial, a contradiction since $\Delta$ is a special cycle.

Thus, we have proved that $u\not\in I^{s+1}.$

We show that $x_j u\in I^{s+1}$ for $1\leq j\leq 2s+1.$

We observe that we may write \[x_1 u=\xb_{F_1}\xb_{F_3}\cdots \xb_{F_{2s-1}}\xb_{F_{2s+1}} w\] for some monomial $w$. This shows that
$x_1 u\in I^{s+1}.$ Let $j\geq 2,$ $j$ even. Then we see that
\[\xb_{F_1}\xb_{F_3}\cdots \xb_{F_{j-1}}\xb_{F_j}\xb_{F_{j+2}}\cdots \xb_{F_{2s}}\] divides $x_ju,$ which shows that $x_ju\in I^{s+1}.$
If $j\geq 2$, $j$ odd, then \[\xb_{F_1}\xb_{F_3}\cdots \xb_{F_{j}}\xb_{F_{j+1}}\xb_{F_{j+3}}\cdots \xb_{F_{2s}}\] divides $x_ju,$  thus
$x_ju\in I^{s+1}.$  Consequently, we obtained $P\subset I^{s+1}:u.$

For the inclusion $I^{s+1}:u\subset P,$ let us consider some monomial $v\in I^{s+1}:u$ and assume that $v\not\in P.$ This means that
none of the variables $x_1,\ldots,x_{2s+1}$ divides $v.$ As $vu\in I^{s+1}$, there exists $G_1,\ldots,G_{s+1}\in \MF(\Delta)$ such that
$w=\xb_{G_1}\cdots \xb_{G_{s+1}}$ divides $vu.$ Then \[\sum_{k=1}^{2s+1}\deg_{x_k}(w)=2s+2=\sum_{k=1}^{2s+1}\deg_{x_k}(vu)=2s+1,\] contradiction. Therefore, $I^{s+1}:u=P,$ which shows that $P\in \Ass(I^{s+1}).$

In order to show that $P\in \Ass(I^m)$ for $m> s+1,$ we may argue in a similar way as above, but using the monomial $w=u(\xb_{F_1})^{m-s-1}$
instead of $u.$

\textit{Step 2.} We show that  if $m\leq s,$ then $P\not\in \Ass(I^m).$ Assume there exists $u\notin I^m$ for some $m\leq s$ such that
$P=I^m:u.$ Then there exists some vertex $j$ of $\Delta$ with  $1\leq j\leq 2s+1,$ such that $x_j$ does not divide $u.$ Indeed, let us assume that $x_1x_2\cdots x_{2s+1}$ divides $u$. But $u$ cannot be divisible by the product of all the variables since, otherwise, we would
get $u\in I^s\subseteq I^m.$ Then there exists some vertices in $\Delta\setminus\{1,2,\ldots,2s+1\},$ say $k_1,\ldots,k_a,$ such that  the product
$x_{k_1}\cdots x_{k_a}$ does not divide  $u.$ Then  $x_{k_1}\cdots x_{k_a} u$ is divisible by the product of all the variables, thus
$x_{k_1}\cdots x_{k_a} u\in I^s\subseteq I^m$ which implies that $x_{k_1}\cdots x_{k_a} \in P,$ contradiction. Consequently, we have proved that there exists some vertex $j$ of $\Delta$ with  $1\leq j\leq 2s+1,$ such that $x_j$ does not divide $u.$

Let $\Delta^\prime$ the simplicial complex obtained from $\Delta$ by removing the facets $F_{j-1}, F_j$ which contain the vertex $j$ and
$I^\prime=I(\Delta^\prime).$ Then $\Delta^\prime$ has no special odd cycle, thus $\Ass((I^\prime)^m)=\Min(I^\prime)$ for all $m\geq 1.$
We will show that $(I^\prime)^m: u=P^\prime$ where $P^\prime$ is the prime ideal obtained from $P$ by removing the generator $x_j.$ This leads to a contradiction since it implies that $P^\prime\in \Ass((I^\prime)^m),$ but, on the other hand, $\{1,2,\ldots,j-1,j+1,\ldots,2s+1\}$ is not a minimal vertex cover of $\Delta^\prime.$

Let $i\neq j,$ $1\leq i\leq 2s+1.$ Then $x_iu\in I^m.$ Since $x_iu$ is not divisible by $x_j,$ it follows that $x_i u\in (I^\prime)^m,$
hence $P^\prime\subset (I^\prime)^m: u.$ Let now assume that $P^\prime\subsetneq (I^\prime)^m: u.$ Then there exists
a minimal  monomial generator $v\in (I^\prime)^m: u$ such that $v\not\in P^\prime.$ As $vu\in (I^\prime)^m\subset I^m,$ it follows that $v\in I^m:u=P. $ Thus we may write $v=x_j^t w$ where $t$ is a positive integer and $w$ is a monomial which is not divisible by $x_j.$
Then we get $x_j^t w u=u v\in (I^\prime)^m,$ thus  $w u\in (I^\prime)^m: x_j^t=(I^\prime)^m.$ This implies that
$w\in (I^\prime)^m:u,$ a contradiction with the choice of $v.$ Consequently, $P^\prime=(I^\prime)^m:u$ and the proof of Step 2 is completed.

\textit{Step 3.} It remains to show that, for every $m\geq 1,$ if $Q$ is an associated prime of $I^m$ different from $P,$ then $Q$ is a minimal prime ideal of $I.$ Since $Q\neq P,$ there exists $1\leq j\leq 2s+1$ such that $x_j\not\in Q.$ Let us consider the localization
of $I$ with respect to the prime ideal $\pp=(x_i: i\neq j, 1\leq i\leq n)$.
Then $I_{\pp}=I^\prime S_\pp$ where
$I^\prime=I(\Delta^\prime)$ and $\Delta^\prime$ is generated by  the facets
\[F_1,\ldots,F_{j-2}, F_{j-1}\setminus\{j\}, F_{j}\setminus\{j\}, F_{j+1},\ldots,F_{2s+1}\] if $|F_{j-1}|, |F_j|\geq 3.$
If $|F_{j-1}|=2, |F_j|\geq 3,$ then $\Delta^\prime$ has two connected components, one of them trivial consisting of  the facet $\{j-1\}$, and the other
one generated by the facets $F_i$ of $\Delta$ with $i\neq j-2,j-1,j$ and the face $F_j\setminus\{j\}\in \Delta.$
If $|F_{j-1}|\geq 3, |F_j|=2,$ then $\Delta^\prime$ consists of the trivial component $\{j+1\}$ and the component generated by the facets $F_i$ of $\Delta$ with $i\neq j-1, j, j+1$ and the face $F_{j-1}\setminus\{j\}\in \Delta.$
Finally, if $|F_{j-1}|=2, |F_j|=2,$ then in $\Delta^\prime$ we have two trivial components, namely $\{j-1\},\{j+1\}$ and a non-trivial component determined by the facets $F_i$ of $\Delta$ with $i\neq j-2,j-1,j,j+1.$

In the first case, namely if $|F_{j-1}|, |F_j|\geq 3,$ $\Delta^\prime$ has no special odd cycle, thus $I^\prime$ is normally torsionfree, which implies that $Q$ is a minimal prime of $I^\prime$ and, clearly, of $I$ as well. In all the other cases, by applying Lemma~\ref{split},
 we derive that $Q$ is a minimal prime of $I$ which does not contain the vertex  $j.$
\end{proof}

\section{Associated prime ideals of $t$-spread principal Borel ideals generated in degree $3$}

The $t$-spread strongly stable ideals with $t\geq1$ have been recently introduced in \cite{EHQ} and they represent a special class of square-free monomial ideals.

Fix a field $\KK$ and a polynomial ring $S=\KK[x_1,\ldots, x_n]$. Let $t$ be a positive integer.
A monomial $x_{i_1}x_{i_2}\cdots x_{i_d}\in S$ with $i_1\leq i_2\leq\ldots \leq i_d$  is called \emph{$t$-spread} if $i_j- i_{j-1}\geq t$ for $2\leq j \leq d$.

A  monomial ideal in $S$ is called  a \emph{$t$-spread monomial ideal}  if it is generated by  $t$-spread monomials. For example,  \[I=(x_1x_4x_9,x_1x_5x_8,x_1x_5x_9,x_2x_6x_9,x_4x_8) \subset \KK[x_1, \ldots, x_9]\] is a $3$-spread monomial ideal, but not $4$-spread, because $x_1x_5x_8$ is not a $4$-spread monomial.

For a monomial $u\in S$, we set \[\supp(u)=\{i:x_i\mid u\}.\]

\begin{Definition}\cite[Definition 1.1]{EHQ}
Let $I$ be a $t$-spread monomial ideal. Then $I$ is called a {\em $t$-spread strongly stable ideal} if for all $t$-spread monomials $u\in G(I)$, all $j\in\supp(u)$ and all $1\leq i<j$ such that $x_i(u/x_{j})$ is a $t$-spread monomial, it follows that $x_i(u/x_j)\in I$.
\end{Definition}

Let $t$ be a positive integer. A monomial  ideal $I\subset S=\KK[x_1,\ldots, x_n]$ is called \emph{$t$-spread principal Borel} if there exists a monomial $u\in G(I)$ such that $I$ is the smallest $t$-spread strongly stable ideal which contains $u.$ According to \cite{EHQ}, we denote $I=B_t(u)$. For example, for an  integer $d\geq 2,$ if $u=x_{n-(d-1)t}\cdots x_{n-t}x_n$, then $B_t(u)$ is minimally generated by all the $t$-spread monomials of degree $d$ in $S.$

Let $u=x_{i_1}x_{i_2}\cdots x_{i_d}\in S= \KK[x_1,\ldots, x_n]$ be a $t$-spread monomial.
Notice that \[x_{j_1}\cdots x_{j_d}\in G(B_t(u))\] if and only if \[j_1\leq i_1,\ldots, j_d\leq i_d\text{ and }j_k-j_{k-1}\geq t\text{ for }k\in\{2,\ldots,d\}.\]

In what follows, we study some classes of $t$-spread principal Borel ideals which are almost normally torsionfree.

The minimal prime ideals of a $t$-spread principal Borel ideal were determined implicitly in the proof of \cite[Theorem 1.1]{AEL}. Therefore, in the following result, we reformulate the statement of \cite[Theorem 1.1]{AEL}.

\begin{Theorem}\label{ass}
  Let $t\geq 1$ be an integer and $I=B_t(u)$, where $u=x_{i_1}x_{i_2}\cdots x_{i_d}\in S$ is  a $t$-spread monomial. We assume that $i_d=n.$ Then every associated prime ideal of $B_t(u)$ is of one of the following forms:
  \begin{equation}\label{form1}
  (x_1,\ldots, x_{j_1-1}, x_{j_1+t}, x_{j_1+t+1},\ldots, x_{j_2-1}, x_{j_2+t},\ldots,x_{j_{d-1}-1}, x_{j_{d-1}+t}, x_{j_{d-1}+t+1}, \ldots,x_n)
\end{equation}
   with $j_l\leq i_l$ for $1\leq l\leq d-1$ and $j_l-j_{l-1}\geq t$ for $2\leq l\leq d-1$.
\begin{equation}\label{form2}
 (x_1,x_2,\ldots,x_{i_1})
\end{equation}
\begin{equation}\label{form3}
(x_1,\ldots, x_{j_1-1}, x_{j_1+t}, x_{j_1+t+1},\ldots, x_{j_2-1}, x_{j_2+t},\ldots,x_{j_{s-1}-1}, x_{j_{s-1}+t}, x_{j_{s-1}+t+1}, \ldots,x_{i_s})
\end{equation}
    with $2\leq s\leq d-1$, $j_l\leq i_l$ for $1\leq l\leq s-1$ and $j_l-j_{l-1}\geq t$ for $2\leq l\leq s-1$.
\end{Theorem}

\begin{Example}
Let $u=x_3x_7x_{10}\in \KK[x_1,\ldots,x_{10}]$ and $I=B_3(u)$. Then we have
\begin{align*}
 \Ass(I)=& \{P_1=(x_7,x_8,x_9,x_{10}), P_2=(x_4,x_8,x_9,x_{10}), P_3=(x_4,x_5,x_9,x_{10}),\\
        & P_4=(x_4,x_5,x_6,x_{10}), P_5=(x_1,x_8,x_9,x_{10}), P_6=(x_1,x_5,x_9,x_{10}), \\
        & P_7=(x_1,x_5,x_6,x_{10}), P_8=(x_1,x_2,x_9,x_{10}), P_9=(x_1,x_2,x_6,x_{10}), \\
        & P_{10}=(x_1,x_2,x_3),\\
        & P_{11}=(x_4,x_5,x_6,x_7),P_{12}=(x_1,x_5,x_6,x_7), P_{13}=(x_1,x_2,x_6,x_7)\}.
\end{align*}
\end{Example}

In the same paper, in \cite[Corollary 2.6]{AEL}, it was also proved that every $t$-spread principal Borel ideal satisfies the persistence property.

\begin{Remark}\label{rem}
Let $u=x_{i_1}\cdots x_{i_{d-1}}x_n$ be a $t$-spread monomial in $S$, where $t$ is a positive integer. We consider $I=B_t(u)$.  If $i_1<t$, then \[\bigcup_{v\in G(I)}\supp(v)=[n]\setminus\{i_1+1, i_1+2, \ldots, t\}\] and $I=B_t(u)$ is in fact an $i_1$-spread ideal in the polynomial ring \[\KK[x_j:j\notin\{i_1+1, i_1+2,\ldots, t\}].\] Therefore, in what follows, we may consider $I=B_t(u)$ to be a $t$-spread principal Borel ideal with $u=x_{i_1}\cdots x_{i_{d-1}}x_n$ and $i_1\geq t$.
\end{Remark}

Let $u=x_ix_n$ be a $t$ spread monomial in $S$, where $t$ is a positive integer. In \cite[Theorem 1.1]{LR}, it was proved that $\Ass^{\infty}(B_t(u))=\Min(B_t(u))\cup(x_1, x_2,\ldots, x_n)$ if $i\geq t+1$ and $B_t(u)$ is normally torsionfree if $i=t$. In the following proposition, we give another proof for \cite[Theorem 1.1]{LR} by using the main result of the first section, Theorem~\ref{graph}.

\begin{Proposition}\label{deg2}
Let $u=x_ix_n$ be a $t$-spread monomial in $S$ with $i\geq t$. Then $I=B_t(u)$ is almost normally torsionfree. Moreover, $I$ is normally torsionfree if and only if $i=t$.
\end{Proposition}
\begin{proof}

  If $i=t$, then $I$ is the edge ideal of a bipartite graph on the vertex set \[\{1,2,\ldots, t\}\cup\{t+1,\ldots, n\}.\] Thus, $I$ is normally torsionfree.

  If $i>t$, then $I$ is the edge ideal of a connected graph $G$ which contains an odd cycle. Let \[C=(\{i_0,i_1\},\{i_1, i_2\},\ldots, \{i_{r-1}, i_0\})\] be an odd cycle of $G$. We consider $i_r=i_0$ and $i_{-1}=i_{r-1}$.

  In the case that there exists $k\in \{0,1,\ldots, r-1\}$ such that
  \begin{equation}\label{tip}
    t+1\leq i_{k-1}\leq i\text{, }i_{k-1}+t\leq i_k\leq i\text{ and }i_k+t\leq i_{k+1},
  \end{equation}
  $I$ is almost normally torsionfree and $\Ass^{\infty}(I)=\Min(I)\cup \{(x_1,\ldots, x_n)\}$.
  Indeed, we have \[1,2,\ldots, t,t+1, \ldots, i_{k-1}\in V(C)\cup N(C)\] because $i_k\geq i_{k-1}+t>2t$.
  Since $i_{k+1}\geq i_k+t\geq i_{k-1}+2t$, we also have \[i_{k-1}+1, i_{k-1}+2, \ldots, i_{k-1}+t-1,i_{k-1}+t, i_{k-1}+t+1, \ldots, n\in V(C)\cup N(C).\]
  Therefore, $V(C)\cup N(C)=[n]$ and by Theorem~\ref{graph}, the proof of this case is completed.

  In the case that for any $k\in\{0,1,\ldots, r-1\}$, the inequalities (\ref{tip}) are not fulfilled, there exists $\ell \in \{0,1,\ldots, r-1\}$ such that $1\leq i_{\ell}\leq t$. Since $C$ is an odd cycle, we may choose $\ell$ such that \[t<i_{\ell+1}\leq i\text{ and }i_{\ell+2}\geq i_{\ell+1}+t.\]
  In this case, \[1,2,\ldots,t, t+1,\ldots, i_{\ell+1}\in V(C)\cup N(C)\] because $i_{\ell+2}\geq i_{\ell+1}+t>2t$.
  We also obtain \[i_{\ell+1}+t, i_{\ell+1}+t+1,\ldots, n\in V(C)\cup N(C).\]
  Since $i_{\ell+1}\geq t+i_{\ell}$, \[i_{\ell+1}+1, i_{\ell+1}+2,\ldots, i_{\ell+1}+t-1\in V(C)\cup N(C).\] Thus, $V(C)\cup N(C)=[n]$ and by Theorem~\ref{graph}, $I$ is again almost normally torsionfree with $\Ass^{\infty}(I)=\Min(I)\cup \{(x_1,\ldots, x_n)\}$.
\end{proof}

\begin{Theorem}\label{tspread}
  Let $t$ be a positive integer and $I=B_t(u)$, where $u=x_{i_1}x_{i_2}x_n$ is a $t$-spread monomial. Then the following condition are equivalents:
  \begin{enumerate}
    \item $B_t(u)$ is normally torsionfree.
    \item $B_t(u)$ is almost normally torsionfree.
    \item $u=x_tx_{2t}x_n$.
  \end{enumerate}
\end{Theorem}
\begin{proof}
  $(1)\Rightarrow (2)$ is trivial.

  $(2)\Rightarrow (3)$: By Remark~\ref{rem}, we consider $I=B_t(u)$ to be a $t$-spread principal Borel ideal with $u=x_{i_1}x_{i_2}x_n$ and $i_1\geq t$.

  If $i_2\geq 2t+1$, then \[P_1=(x_{t+1}, x_{t+2},\ldots, x_n)\text{ and }P_2=(x_1,x_2,\ldots, x_{t-1}, x_{2t}, x_{2t+1}, \ldots, x_n)\] belong to $\Ass(I^2)\setminus \Ass(I)$. Indeed, we have \[I_{P_1}=B_t(x_{i_2}x_n)\subset S_{P_1}\text{ and }I_{P_2}=B_t(x_{i_2}x_n)\subset S_{P_2},\]
  where $I_{P_j}$ is the localization of $I$ with respect to $P_j$ for $j\in\{1,2\}$.
  Since $i_2\geq 2t+1$, we obtain $\depth(S_{P_1}/I_{P_1}^k)=0$ and $\depth(S_{P_2}/I_{P_2}^k)=0$ for every $k\geq 2$, by \cite[Theorem 3.1]{AEL}. It implies that $P_1, P_2\in \Ass(I^2)$. According to Theorem~\ref{ass}, $P_1, P_2\notin\Ass(I)=\Min(I)$.

  Since $I$ is almost normally torsionfree, we must have $i_2\leq 2t$. Thus, we obtain $i_1\geq t$, $i_2\leq 2t$ and $i_2-i_1\geq t$. In other words, $i_1=t$ and $i_2=2t$.

  $(3)\Rightarrow (1)$: We consider $u=x_tx_{2t}x_n$ with $n\geq 3t$. Suppose that $B_t(u)$ is not normally torsionfree. Then there exist $k\geq 2$ and $P\in \Ass(I^k)\setminus \Ass(I)$. Let $w$ be a monomial in $S$ with the property that $w\notin I^k$ and $P=I^k:w$.

  If $x_1\notin P$, then $P\in \Ass(I_{(x_2,\ldots, x_n)}^k)$. Since \[I_{(x_2,\ldots, x_n)}=B_t(x_{2t}x_n)\subset\KK[x_{t+1}, x_{t+2},\ldots, x_n]\] is normally torsion free, $P\in\Ass(I_{(x_2,\ldots, x_n)})$. Hence, we have $P\in\Ass(I)$ which is false. So, $x_1\in P$. It implies that $x_1\cdot w=u_1u_2\cdots u_k$, where $u_j\in G(I)$ for every $1\leq j\leq k$. By \cite[Proposition 3.1]{EHQ}, $G(I)$ is sortable. Therefore, we may consider the $k$-tuple $(u_1,\ldots, u_k)$ to be sorted. It implies that $x_1\mid u_1$ and $w=(u_1/x_1)u_2\cdots u_k$.

  We also have $x_n\in P$. Otherwise, $P\in \Ass(I_{(x_1,\ldots, x_{n-1})}^k)$. Since \[I_{(x_1,\ldots, x_{n-1})}=B_t(x_{t}x_{2t})\subset\KK[x_{1}, x_{2},\ldots, x_{2t}]\] is normally torsionfree, $P\in\Ass(I_{(x_1,\ldots, x_{n-1})})$. Thus, we obtain $P\in\Ass(I)$ which is false.

  Because $x_n\in P$, \[x_n\cdot w=(x_nu_1/x_1)u_2\cdots u_k\in I^k.\] Notice that a monomial $v\in I^k$ if and only if the monomial $v$ contains at least $k$ variables with indices less than $t+1$, $k$ variables with indices in the interval $[t+1,2t]$ and $k$ variables with indices greater than $2t$. In the product $x_n\cdot w$, we have only $k-1$ variables with indices less than $t+1$. It implies that $x_n\cdot w\notin I^k$ which is a contradiction.
\end{proof}

\end{document}